\documentclass{amsart}
\usepackage[nobysame]{amsrefs}
\usepackage{tikz}
\usepackage{graphicx}
\usepackage{setspace}
\usepackage{amssymb}
\usepackage{fullpage}
\usepackage{comment}
\usepackage[T1]{fontenc}
\usepackage{setspace}

\newtheorem{theorem}{Theorem}
\newtheorem{lemma}{Lemma}
\newtheorem{definition}{Definition}
\title{An Approach to Studying Quasiconformal Mappings on Generalized Grushin Planes}
\author{Colleen Ackermann}
\address{University of Illinois at Urbana-Champaign Department of Mathematics, 1409 W. Green St, Urbana, IL 61801, USA}
\email{ackrmnn2@illinois.edu}
\thanks{The author acknowledges support from NSF Grant DMS 08-38434 "EMSW21-MCTP: Research
Experiences for Graduate Students". }
\date{\today}

\begin{document}

\maketitle
{
\begin{abstract} We demonstrate that the complex plane and a class of generalized Grushin planes $G_r$, where $r$ is a function satisfying specific requirements, are quasisymmetrically equivalent.  Then using conjugation we are able to develop an analytic definition of quasisymmetry for homeomorphisms on $G_r$ spaces.  In the last section we show our analytic definition of quasisymmetry is consistent with earlier notions of conformal mappings on the Grushin plane.  This leads to several characterizations of conformal mappings on the generalized Grushin planes.
\end{abstract}


%%%%%%%%%%%%%%%%%%%% Introduction

\section{Introduction}
The concept of a quasiconformal mapping in the complex plane was originally formulated by Gr\"{o}tzsch in 1928 \cite{MR0200442}.  Intuitively, a quasiconformal mapping is a homeomorphism that maps infinitesimal circles to infinitesimal ellipses.  More precisely we have the following definition:
\begin{definition}\label{MetricDefn}
 Let $f$ be a homeomorphism between two domains in $\mathbb{C}$.
 Define
 $$ L_f(z,R)=\sup_{|z-z'|=R} |f(z)-f(z')|,$$

 $$ l_f(z,R)=\inf_{|z-z'|=R} |f(z)-f(z')|$$
 and $$H_f(z,R)=\frac{L_f(z,R)}{l_f(z,R)}.$$
 Then $f$ is quasiconformal provided that $$\limsup_{R\to 0} H_f(z,R)$$ is uniformly bounded for all $z$ in its domain.
 \end{definition}

As the theory was developed in the complex plane many equivalent definitions of quasiconformality arose.  One of these which will appear in Section \ref{AnalyticDefnSection} is the analytic definition.

\begin{definition}\label{CAnalyticDefn}
A homeomorphism $f:A\to B$, where $A$ and $B$ are domains in $\mathbb{C}$, is absolutely continuous on lines if for every rectangle $R=\{(x,y): a < x < b, c < y < d\}$, $\overline{R}\subset A$, $f$ is absolutely continuous on a.e. interval $I_x=\{(x,y): c<y<d\}$ and a.e. interval $I_y=\{(x,y):a<x<b\}$.  
Suppose $f$ is absolutely continuous on lines and satisfies 
\begin{equation}
 \frac{\partial f}{\partial \bar{z}}=\mu \frac{\partial f}{\partial z} \text{ a.e.}
 \end{equation}
 where $\mu$ is some measurable function with $||\mu||_{\infty}<1$.  Then $f$ is quasiconformal, and we say $f$ satisfies the (classical) Beltrami equation.
\end{definition}

Yet another condition called quasisymmetry is equivalent to quasiconformality in some cases, but not others.

\begin{definition}
Let $A$ and $B$ be domains in $\mathbb{C}$ and $f: A \to B$ a homeomorphism.  We say $f$ is quasisymmetric if there exists a homeomorphism $\eta : [0,\infty)\to [0,\infty)$ such that for all triples of points $a,b,c\in A$ we have
 $$|a-b|\leq t |b-c| \implies |f(a)-f(b)|\leq \eta(t) |f(b)-f( c)|.$$
\end{definition}

Quasisymmetric maps are always quasiconformal, but quasiconformal maps are not always quasisymmetric.  However, the two concepts are equivalent for homeomorphisms of the complex plane.  We will make considerable use of this fact in Section \ref{AnalyticDefnSection}.  
\\ \indent Quasisymmetric maps are generally easier to work with.  For example, it is simple to prove the composition of two quasisymmetric maps is quasisymmetric, and inverses of quasisymmetric maps are quasisymmetric.  However, it is far from trivial to see from Definition \ref{MetricDefn} that the set of quasiconformal maps of the plane is closed under composition and inversion.
\\ \indent
The first definition we gave of quasiconformality and our definition of quasisymmetry make sense on general metric spaces so naturally the theory did not remain only in the complex plane.  First the theory was expanded to $\mathbb{R}^n$ \cite{MR0454009}. 
Then later it was extended to the Heisenberg group \cite{MR788413}, general Carnot groups \cite{MR0385004} \cite{MR1260606} and finally equiregular sub-Riemannian manifolds \cite{MR1334873}, and Ahlfors regular metric measure spaces \cite{MR1654771}.  However, the theory has been largely unexplored for metric spaces that are non-Ahlfors regular.
The simplest example of such a space is the Grushin plane.
\\
\begin{definition} 
The (classical) Grushin plane $G$ is $\mathbb{R}^2$ with the metric defined by the Carnot-Carath\'{e}odory distance  
$$d_{CC}(w,w')=\inf\ell(\gamma)$$
where the infimum is taken over all absolutely continuous, horizontal paths $\gamma=(\gamma_1,\gamma_2):[0,1]\to G$ with $\gamma(0)=w$ and $\gamma(1)=w'$, and the length
$$\ell(\gamma)=\ell(\gamma_1,\gamma_2)=\int_0^1\sqrt{(\gamma_1'(s))^2+\frac{(\gamma_2'(s))^2}{(\gamma_1(s))^2}}ds.$$
Paths on the Grushin plane are horizontal if they have a horizontal tangent at the vertical axis.
\end{definition}
	Throughout this paper we take $(u,v)$ to be the coordinates on the Grushin plane.
	The Grushin plane is Riemannian everywhere except on the singular line $u=0$.  The metric for the Grushin plane is defined using the vector fields $\frac{\partial}{\partial u}$ and $|u|\frac{\partial}{\partial v}$ which span the entire tangent space except along the vertical axis which is sub-Riemannian by Chow's condition \cite{MR1421822}.  In Section 2 we will see easily computable estimates for the Carnot-Carath\'{e}odory distance, and discuss the non-Ahlfors regularity of the Grushin plane.
\\ \indent	William Meyerson showed the complex plane and the Grushin plane are quasisymmetrically equivalent via the map $(u,v)\to u|u|+iv$.  He then generalized this result and showed metric spaces defined by the vector fields $\frac{\partial}{\partial u}$ and $|u|^{\alpha}\frac{\partial}{\partial v}$ where $\alpha>0$, are quasisymmetrically equivalent to the complex plane \cite{Meyerson}.  In this paper we determine conditions on a homeomorphism $r:\mathbb{R}\to\mathbb{R}$ such that $\Phi_r:(u,v)\to r(u)+iv$ is a quasisymmetry between the complex plane and the metric space $G_r$ defined by the vector fields $\frac{\partial}{\partial u}$ and $r'(u)\frac{\partial}{\partial v}$.  These quasisymmetries are of interest to us, because they can be used to translate the rich theory of quasiconformal mappings in the complex plane to the $G_r$ spaces via conjugation.  For example, we can define the $r$-Grushin Beltrami equation as follows:

\begin{definition}\label{rGrushinBeltramiEquation} Suppose $g=(g_1,g_2):G_r\to G_r$ and define $\tilde{g}=\Phi_r\circ g$, $W=\frac{1}{2}(\frac{\partial}{\partial u}-ir'(u)\frac{\partial}{\partial v})$ and $\overline{W}=\frac{1}{2}(\frac{\partial}{\partial u}+ir'(u)\frac{\partial}{\partial v})$.  We say $\tilde{g}$ satisfies the $r$-Grushin Beltrami equation provided that
 \begin{equation}\label{rGrushinBeltrami}
  \overline{W} \tilde{g}=\nu W \tilde{g} \text{ a.e.}
  \end{equation}
 where $\nu$ is some measurable function with $||\nu||_{\infty}<1$.
 \end{definition}
 
 Then we obtain an analytic characterization of quasisymmetry in $G_r$.
 
 \begin{theorem}\label{QuasisymAnalyticDefn} A map $g: G_r\to G_r$ is quasisymmetric if and only if $\tilde{g}$ is a homeomorphism that is absolutely continuous on lines, and satisfies equation \ref{rGrushinBeltrami} for all points at which it is defined.  
 \end{theorem}
 We will clarify what we mean for $\tilde{g}$ to be absolutely continuous on lines in Section \ref{AnalyticDefnSection}.
\\ \indent  In the last section we will seek to reconcile this theorem with notions of conformal mappings.  For example, we will generalize the definition of conformality on Riemannian manifolds to develop the following definition:
  
  \begin{definition} Suppose $A$ and $B$ are domains in $G$ and $g=(g_1,g_2): A \to B$ is a homeomorphism.  Define $A'=A-\{u=0 \text{ or } g_1(u)=0\}$.  We say $g|_{A'}$ is conformal provided that 

$$D g=\left( \begin{array} {rr}
 \frac{\partial g_1}{\partial u} & |u|\frac{\partial g_1}{\partial v} \\
 \frac{1}{|g_1|}\frac{\partial g_2}{\partial u} & \frac{|u|}{|g_1|}\frac{\partial g_2}{\partial v}
 \end{array} \right)$$
is defined and is a conformal matrix for every point in $A'$.  We say $g$ is conformal on all of $A$ if $g$ is conformal on $A'$ and for all points $w_0\in A-A'$, $$\lim_{w\to w_0} D_rg(w)$$ is defined and non-zero.  We take the limit along all paths in $A'$.
\end{definition}

 We will show with certain conditions this definition is equivalent to $g$ being quasisymmetric and $\nu$ being identically zero.  Furthermore, our definition is satisfied by a class of conformal maps on the Grushin plane discovered by Payne \cite{MR2230585}.
 
\indent  We hope our study of conformal mappings on the Grushin plane will help us to develop a geometric definition of quasiconformal mappings (\cite{MR0200442} p. 21).  If this is possible, we plan to compare the geometric definition with the analytic and metric definitions, and determine when and in what ways they are equivalent.  We would like to replicate as much as possible of the theory of quasiconformal mappings in the complex plane.  Eventually we hope this work will provide insights into the theory of quasiconformal mappings on other non-equiregular spaces. 

\section*{Acknowledgements}

The author would like to thank Jeremy Tyson and the referee for their many helpful comments and suggestions.

%%%%%%%%%%%%%%%%%% Basic Geometry of the G Spaces

 \section{Basic Geometry of the $G_r$ Spaces}

Before proving quasisymmetry we must develop a basic picture of the geometry of the $G_r$ spaces.  To the best of our knowledge the following definition is original.

\begin{definition}\label{GeneralizedGrushin}
Let $r:\mathbb{R}\to\mathbb{R}$ be a differentiable homeomorphism satisfying the following properties:
\\ 1. $r'$ is an even function and $r'|_{[0,\infty)}$ is a homeomorphism onto itself.
\\ 2.  There exists $\beta>1$ such that for all $u\in\mathbb{R}-\{0\}$
$$\frac{r(u)}{u} \leq r'(u)\leq \beta \frac {r(u)}{u}.$$
The $r$-Grushin plane $G_r$ is $\mathbb{R}^2$ with the metric defined by the Carnot-Carath\'{e}odory distance  
$$d_{rCC}(w,w')=\inf\ell_r(\gamma)$$
where the infimum is taken over all absolutely continuous, horizontal paths $\gamma=(\gamma_1,\gamma_2):[0,1]\to G_r$ with $\gamma(0)=w$ and $\gamma(1)=w'$, and the length
$$\ell_r(\gamma)=\ell_r(\gamma_1,\gamma_2)=\int_0^1\sqrt{(\gamma_1'(s))^2+\frac{(\gamma_2'(s))^2}{(r'(\gamma_1(s)))^2}}ds.$$
\end{definition}
Just as on the classical Grushin plane, paths on $G_r$ are horizontal if they have a horizontal tangent at the vertical axis.
\\ 
\indent The simplest example of homeomorphisms $r$ satisfying Definition \ref{GeneralizedGrushin} are the power functions used by Meyerson which were mentioned in the introduction.  Another slightly more complex class of examples are functions of the form
$$ r(u)=\left\{
\begin{array}{lr}
u^p\ln(u+1) & u\geq 0 \\
-|u|^p\ln(|u|+1) &  u<0
\end{array}
\right.
$$
where $p>1$.
The reader can easily check that these satisfy the requirements of our definition.
\\
\indent The following lemma gives another property of $r$ and will be used throughout our proof of quasisymmetry.

 \begin{lemma}\label{doubling}
As defined above the function $r'$ is doubling when restricted to $[0,\infty)$.  In other words, there exists a constant $m>0$ such that for all $u\in[0,\infty)$ we have $r'(2u)\leq mr'(u)$.
\end{lemma}

\begin{proof} First we show $r|_{(0,\infty)}$ is doubling.  Choose $\alpha>1$ such that $\beta\ln\alpha<1$.  By our conditions on $r$ we have
$r(\alpha u)=\int_u^{\alpha u} r'(t)dt+r(u)
\leq \int_u^{\alpha u} \beta \frac {r(t)}{t}dt + r(u)
\leq \beta r(\alpha u)\int_u^{\alpha u}\frac{dt}{t}+r(u)
= \beta r(\alpha u)\ln\alpha +r(u).$
  Thus $r(\alpha u)\leq \frac{r(u)}{1-\beta\ln\alpha}$ where $\frac{1}{1-\beta\ln\alpha}>0$.  Since $\alpha>1$ and $r|_{(0,\infty)}$ is increasing, repeated iteration gives $r|_{(0,\infty)}$ is doubling for some constant $m$.  Then since
  $$ \frac{2 u}{\beta}r'(2u)  \leq r(2u)\leq mr(u)\leq mu r'(u),$$
  $r'$ restricted to $(0,\infty)$ is also doubling.  The claim is trivial for $u=0$.
\end{proof}

Since the Carnot-Carath\'{e}odory distance does not lend itself to proving quasisymmetry directly we will define a quasidistance $d_r$, and then show it suffices to only consider the quasidistance.  More precisely, we will show there exists a constant $C$ such that if $w,a,b\in G_r$ and $d_{rCC}(w,a)\leq d_{rCC}(w,b)$, then $d_r(w,a)\leq C d_r(w,b)$.   The definition below is a generalization of Meyerson's quasidistance \cite{Meyerson}.
\\ 
\begin{definition}
The $r$-Grushin quasidistance between two points $w,w'\in G_r$ is
$$d_r(w,w')=\max\left\{|u-u'|,\min\left\{M,\frac{|v-v'|}{\max\{r'(u),r'(u')\}}\right\}\right\}$$
where $M=M(v,v')$ is the unique solution to the equation $M=\frac{|v-v'|}{r'(M)}$.  If $u=u'=0$, and hence $\frac{|v-v'|}{\max\{r'(u),r'(u')\}}$ is undefined, we adopt the convention $d_r(w,w')=M$.
\end{definition}

From now on we simplify our notation by writing $\ell$ for $\ell_r$, $d(w,w')$ for $d_r(w,w')$, $d_{CC}(w,w')$ for $d_{rCC}(w,w')$, $\Phi$ for $\Phi_r$, and $G$ for $G_r$.  Most of what follows is true for all $r$-Grushin planes.  We will clearly state when this is not the case and a result or example applies only to the classical Grushin plane where $r(u)=\frac{1}{2}u|u|$.
\\
\indent The next lemma demonstrates that the Carnot-Carath\'{e}odory metric and the quasidistance on the $r$-Grushin plane are comparable.

 \begin{lemma}\label{Quasidistance}
There exists a positive constant $C$ such that for any two points $w,w'\in G$ 
$$ \frac{1}{C}d_{CC}(w,w')\leq d(w,w')\leq Cd_{CC}(w,w').$$
 \end{lemma}

\begin{proof}
 Let $w=(u,v)$ and $w'=(u',v')$ be points in $G$.  We make use of the following facts which the reader can easily verify:
\begin {enumerate}
\item $d_{CC}((u,v),(u',v))=|u-u'|$.
\item $d_{CC}((u,v),(u,v'))\leq\frac{|v-v'|}{r'(u)}$ provided $u\neq 0$.
\end{enumerate}
First we show $\frac{1}{C}d_{CC}(w,w')\leq d(w,w')$.  If $v=v'$, then $d(w,w')=|u-u'|=d_{CC}(w,w')$.  Hence we may assume $v\neq v'$.   
\\
\\
 {\bf Case 1: $M\geq\frac{|v-v'|}{\max\{r'(u),r'(u')\}}$}
\\  By our convention we can assume either $u$ or $u'$ is nonzero.  Without loss of generality we take $|u|\geq |u'|$ which gives us

\begin{eqnarray*}
d_{CC}(w,w'))
&\leq& d_{CC}((u,v),(u,v'))+d_{CC}((u,v'),(u',v')) \\
&\leq& |u-u'|+\frac{|v-v'|}{r'(u)} \\
&\leq&|u-u'|+\frac{|v-v'|}{\max\{r'(u),r'(u')\}} \\
&\leq& 2\max\left\{|u-u'|,\frac{|v-v'|}{\max\{r'(u),r'(u')\}}\right\} \\
&=& 2d(w,w').
\end{eqnarray*}
\\
\noindent {\bf Case 2: $M\leq\frac{|v-v'|}{\max\{r'(u), r'(u')\}}$}
\\ Then $\max\{r'(u),r'(u')\}\leq r'(M)$, and since $r'$ is even and $r'|_{[0,\infty )}$ is increasing, we have $\max\{|u|,|u'|\}\leq M$.  Thus 
\begin{eqnarray*}
d_{CC}(w,w') 
&\leq& d_{CC}(w,(M,v))+d_{CC}((M,v),(M,v'))+d_{CC}((M,v'),w')\\
&\leq& |M-u|+|M-u'|+\frac{|v-v'|}{r'(M)}\\
&\leq& 5M\\
&\leq& 5d(w,w'). 
\end{eqnarray*}
\\
 This proves $\frac{1}{C}d_{CC}(w,w')\leq d(w,w')$ for some constant $C$.
\\

 \indent To prove $d(w,w')\leq Cd_{CC}(w,w')$ it suffices to show for an arbitrary path $\gamma$ from $w$ to $w'$, $\ell(\gamma)\geq\frac{1}{2m}d(w,w')$.  Recall $m$ is the doubling constant defined in Lemma \ref{doubling}.  We once again assume $|u|\geq |u'|$.  Fix $\gamma=(\gamma_1,\gamma_2)$ and let $s_0$ be such that $|\gamma_1(s)-u|\leq|\gamma_1(s_0)-u|$ for all $s$.  If $|\gamma_1(s_0)-u|\geq d(w,w')$, then since $\ell(\gamma)\geq |\gamma_1(s_0)-u|$, we have our desired inequality.  Now we assume $|\gamma_1(s_0)-u|< d(w,w')$.  In other words

$$|\gamma_1(s_0)-u|<\max\left\{|u-u'|,\min\left\{M,\frac{|v-v'|}{\max\{r'(u),r'(u')\}}\right\}\right\}.$$

By the definition of $s_0$, $|\gamma_1(s_0)-u|\geq|u-u'|$ and thus $|\gamma_1(s_0)-u|<M$.  Also by the definition of $s_0$, $|\gamma_1(s)|\leq |u|+|\gamma_1(s_0)-u|$ for all $s$.  Combining our inequalities gives $|\gamma_1(s)|<M+|u|$ for all s.  Then

\begin{eqnarray*}
\ell(\gamma)
&=& \int_0^1\sqrt{(\gamma_1'(s))^2+\frac{(\gamma_2'(s))^2}{(r'(\gamma_1(s)))^2}}ds \\
&\geq& \int_0^1\sqrt{(\gamma_1'(s))^2+\frac{(\gamma_2'(s))^2}{(r'(|u|+M))^2}}ds \\
&\geq& \frac{1}{2}\left(|u-u'|+\frac{|v-v'|}{r'(|u|+M)}\right) \\
&\geq& \frac{1}{2}\left(|u-u'|+\frac{|v-v'|}{r'(2\max\{M,|u|\})}\right) \\
&\geq& \frac{1}{2}\left(|u-u'|+\frac{|v-v'|}{mr'(\max\{M,|u|\})}\right)   \text{ by Lemma \ref{doubling}}\\
&=& \frac{1}{2}\left(|u-u'|+\frac{1}{m}\min\left\{\frac{|v-v'|}{r'(M)},\frac{|v-v'|}{r'(u)}\right\}\right) \\
&\geq& \frac{1}{2m}d(w,w')  \text{  by the definition of $M$}. \\
\end{eqnarray*}
\end{proof}

Before ending this section we would like to return to another geometric feature of the Grushin plane which was briefly mentioned in the introduction.
\begin{definition}
Suppose a metric space $(X,d)$ has Hausdorff dimension $n$.  Let $m$ be the Hausdorff $n$-measure on $X$ and $B(x,R)$ be the ball of radius $R$ centered at $x$.  Then $X$ is called Ahlfors $n$-regular if there exists some constant $C$ such that 
 $$\frac{1}{C}R^n\leq m(B(x,R))\leq CR^n$$
 for all $x\in X$ and all $R>0$.
\end{definition}
The Grushin plane has Hausdorff dimension 2, and any compact subset of the Grushin plane that excludes the singular line is Ahlfors 2-regular.  However, once we include the singular line in our space this fails to be true.  Indeed, for small $\epsilon$, the number of balls of radius $\epsilon$ needed to cover a square with one side on the singular line has magnitude  $\asymp\epsilon^{-2}\ln(\epsilon^{-1})$ \cite{MR1421822}.
\\ \indent This complicates our study of quasiconformal mappings because it is difficult to determine if quasisymmetry and quasiconformality are equivalent.  Quasisymmetry is a global condition while quasiconformality is a local one.  Thus to prove quasiconformality implies quasisymmetry we need some global regularity condition on the geometry of our space.  We would like to have the equivalence of these two definitions, because it is easier to show a function satisfies the conditions for being quasiconformal, but it is usually simpler to prove theorems about quasisymmetric maps.  

%%%%%%%%%%%%%%%%% QS Equivalence of C and G

\section{The Quasisymmetric Equivalence of the Complex Plane and Generalized Grushin Planes}
We begin this section by showing that to prove the $G_r$ spaces and the complex plane are quasisymmetrically equivalent, it suffices to consider the quasidistance instead of the Carnot-Carath\'{e}odory metric on the generalized Grushin planes.  Then we will prove two key lemmas which show how the quasidistance between two points in $G_r$ compares to the distance between their images in the complex plane.  These will allow us to finally prove the desired quasisymmetry with relative ease.  

 \begin{lemma} \label{QuasidistanceSuffices} If $w,a,b\in G$ are such that $d_{CC}(w,a)\leq d_{CC}(w,b)$, then $d(w,a)\leq C^2d(w,b)$.
\end{lemma}

\begin{proof}
 By our previous lemma, $d(w,a)\leq Cd_{CC}(w,a)\leq Cd_{CC}(w,b)\leq C^2 d(w,b)$. \end{proof}

Recall the map $\Phi:G\to\mathbb{C}$ by 
\begin{equation}\label{quasisymmetry}
\Phi(u,v)=r(u)+iv.
\end{equation}
 We will eventually show $\Phi$ is a quasisymmetry.  Throughout our proof we will use the sup norm on $\mathbb{C}$ so $|\Phi(w)-\Phi(w')|=|(r(u)-r(u'),v-v')|=\max\{|r(u)-r(u')|,|v-v'|\}$.  The following two lemmas describe how $d(w,w')$ compares to $|\Phi(w)-\Phi(w')|$.  Note the dependence on the relative magnitudes of $d(w,w')$ and the maximum distance of $w$ and $w'$ from the v-axis.  This is unsurprising since the amount by which the metric on the Grushin plane is distorted from the Euclidean metric depends on a comparison between the same two quantities.

\begin{lemma} \label{Max>Distance}
 Suppose $w,w'\in G$ and $\max\{|u|, |u'|\}\geq d(w,w')$.  Then for some constant $C_1$
			$$\frac{1}{C_1} |\Phi(w)-\Phi(w')|\leq d(w,w')\max\{r'(u),r'(u')\}\leq C_1|\Phi(w)-\Phi(w')|.$$
\end{lemma}

\begin{proof}
Fix $w$, $w'$ such that $\max\{|u|, |u'|\}\geq d(w,w')$.  Then $\max\{|u|, |u'|\}\geq |u-u'|$, and thus $uu'\geq 0$.  By the Mean Value Theorem and our conditions on $r$, for some $c$ between $u$ and $u'$ we have $$|r(u)-r(u')|
=|u-u'|r'(c)
\leq |u-u'|\max\{r'(u),r'(u')\}
\leq |u-u'|\beta\max\left\{\frac{r(u)}{u},\frac{r(u')}{u'}\right\}
\leq\beta|r(u)-r(u')|.$$  The last inequality holds since our conditions on $r$ imply the function $\frac{r(u)}{u}$ is increasing.

  Also if $M\leq\frac{|v-v'|}{\max\{r'(u),r'(u')\}}$, then $r'(M)\leq r'(d(w,w'))\leq\max\{r'(u),r'(u')\}\leq r'(M)$ which implies $r'(M)= \max\{r'(u),r'(u')\}$ and thus $M=\frac{|v-v'|}{\max\{r'(u),r'(u')\}}$.  Therefore we may assume $M\geq\frac{|v-v'|}{\max\{r'(u),r'(u')\}}$ and the lemma follows. \end{proof}
 
\begin{lemma} \label{Max<Distance}
 Suppose $w,w'\in G$ and $\max\{|u|, |u'|\}\leq d(w,w')$.  Then for some constant $C_2$,
			$$ \frac{1}{C_2} |\Phi(w)-\Phi(w')|\leq r'(d(w,w'))d(w,w')\leq C_2|\Phi(w)-\Phi(w')|.$$	
\end{lemma}

\begin{proof}
 Fix $w$, $w'$ such that $\max\{|u|, |u'|\}\leq d(w,w')$.  If $d(w,w')=\frac{|v-v'|}{\max\{r'(u),r'(u')\}}$, then $$r'(M)\leq \max\{r'(u),r'(u')\}\leq r'(d(w,w'))=r'\left(\frac{|v-v'|}{\max\{r'(u),r'(u')\}}\right)\leq r'(M)$$ which implies $M=\frac{|v-v'|}{\max\{r'(u),r'(u')\}}$.  Thus $d(w,w')=\max\{|u-u'|,M\}$.  
 \\ \indent We also have $$ \frac{1}{2}|r(u)-r(u')|\leq \max\{|r(u)|,|r(u')|\}\leq\max\{r'(u) |u|,r'(u')|u'|\}\leq  r'(d(w,w'))d(w,w').$$  Furthermore by our hypothesis, if $d(w,w')=|u-u'|$, we must have $uu'\leq0$ which implies 

\begin{eqnarray*}
r'(d(w,w'))d(w,w')
&=& r'(u-u')|u-u'| \\
&\leq& r'(2\max\{|u|,|u'|\})|u-u'| \\
&\leq& m\max\{r'(u),r'(u')\}|u-u'| \text{ by Lemma \ref{doubling}} \\
&\leq& m(r'(u)+r'(u'))|u-u'| \\
&=& m|ur'(u)-u'r'(u')+ur'(u')-u'r'(u)| \\
&\leq& 2m|ur'(u)-u'r'(u')| \\
&\leq& 2 m \beta |r(u)-r(u')| \\
\end{eqnarray*}
where the last inequality holds because $uu'\leq 0$.  The result follows. \end{proof}

Now we are able to show $\Phi$ is a quasisymmetry.  We actually only prove weak quasisymmetry, but this is equivalent to quasisymmetry for the spaces we are considering.  

\begin{definition}
Let $(X,d_X)$ and $(Y,d_Y)$ be metric spaces.  A homeomorphism $\Phi:X\to Y$ is weakly quasisymmetric if there exists a constant $C$ such that for all triples of points $a,b,c\in X$ we have
 $$d_X(a,b)\leq d_X(b,c) \implies d_Y(\Phi(a),\Phi(b))\leq C d_Y(\Phi(b),\Phi( c))$$
\end{definition}

For a proof of the equivalence of weak quasisymmetry and quasisymmetry the reader is referred to Theorem 10.15 in \cite{MR1800917}.
\\
\begin{theorem} Suppose $a$, $b$ and $w$ are points in the $r$-Grushin plane such that $d_{CC}(w,a)\leq d_{CC}(w,b)$.  Then for some constant $C( r)$ we have
			$$ |\Phi(w)-\Phi(a)|\leq C(r) |\Phi(w)-\Phi(b)|.$$

\end{theorem}

\begin{proof}
 Fix $a,b,w\in G$ such that $d_{CC}(w,a)\leq d_{CC}(w,b)$.  Then $d(w,a)\leq C^2d(w,b)$ by Lemma \ref{QuasidistanceSuffices}.  We divide the proof into the following four cases:
\\
\\ {\bf Case 1: $\max\{|u|,|a_1|\}\leq d(w,a)$ and $\max\{|u|,|b_1|\}\leq d(w,b)$}
\\ By Lemma \ref{Max<Distance}, $$|\Phi(w)-\Phi(a)|\leq C_2r'(d(w,a))d(w,a)\leq C_2C^2r'(C^2d(w,b))d(w,b)\leq C'|\Phi(w)-\Phi(b)|$$ where $C'$ is such that $C_2^2C^2r'(C^2t)\leq C'r'(t)$.  Such a $C'$ can be found since $r$ is doubling.
\\
\\ {\bf Case 2: $\max\{|u|,|a_1|\}\leq d(w,a)$ and $\max\{|u|,|b_1|\}\geq d(w,b)$}
\\ This case is the same as Case 1 except one should use Lemma \ref{Max>Distance} instead of Lemma \ref{Max<Distance} at the end of the chain of inequalities. 
\\
\\
\indent The last two cases are slightly more complicated since first we must find ways to compare $\max\{|u|,|a_1|\}$ with $\max\{|b_1|,|u|\}$ and $d(w,b)$.  After these inequalities are obtained, the proofs follow similarly to those of the first two cases. 
\\
\\ {\bf Case 3: $\max\{|u|,|a_1|\}\geq d(w,a)$ and $\max\{|u|,|b_1|\}\geq d(w,b)$}
\\ Since $d(w,a)\leq C^2d(w,b)$, we have 
$|a_1-u|\leq C^2d(w,b)$ and therefore
$ |a_1|\leq |u|+C^2d(w,b).$
Then we can obtain our desired comparison: $$\max\{|u|,|a_1|\}\leq\max\{|u|,|b_1|\}+C^2d(w,b)\leq (1+C^2)\max\{|u|,|b_1|\}.$$
Finally we have
\begin{eqnarray*}
|\Phi(w)-\Phi(a)|
&\leq& C_1d(w,a)\max\{r'(u),r'(a_1)\} \\
&\leq& C_1C^2d(w,b)r'((1+C^2)\max\{|b_1|,|u|\}) \\
&\leq& C''|\Phi(w)-\Phi(b)| 
\end{eqnarray*}
where $C''$ is such that $C_1^2C^2r'((1+C^2)t)\leq C''r'(t)$.
\\
\\ {\bf Case 4: $\max\{|u|,|a_1|\}\geq d(w,a)$ and $\max\{|u|,|b_1|\}\leq d(w,b)$}
\\ Similarly to the previous case $d(w,a)\leq C^2d(w,b)$ implies $$\max\{|u|,|a_1|\}\leq\max\{|u|,|b_1|\}+C^2d(w,b)\leq (1+C^2)d(w,b).$$
Thus 
\begin{eqnarray*}
|\Phi(w)-\Phi(a)|
&\leq& C_1d(w,a)\max\{r'(u),r'(a_1)\}\\
&\leq& C_1C^2r'((1+C^2)d(w,b))d(w,b) \\
&\leq& C'''|\Phi(w)-\Phi(b)|
\end{eqnarray*}
where $C'''$ is such that $C_1C_2C^2r'((1+C^2)t)\leq C'''r'(t)$.
\end{proof}

Since we have shown that the $r$-Grushin plane is quasisymmetrically equivalent to $\mathbb{C}$, we may ask whether all of our restrictions on the homeomorphism $r$ were necessary.  The requirement that $r'$ is even can almost certainly be eliminated, since it is mostly used to simplify the proof when dealing with $w$ and $w'$ on opposite sides of the v-axis.  The following theorem demonstrates that the other major constraint on $r$ is a necessary condition.  

\begin{theorem} Let $r:\mathbb{R}\to\mathbb{R}$ be a differentiable homeomorphism such that $r'|_{[0,\infty)}$ and $r'|_{(-\infty,0]}$ are homeomorphisms onto $[0,\infty)$ and $(-\infty,0]$ respectively, $r(0)=0$, and $\Phi$, as defined in (\ref{quasisymmetry}) is quasisymmetric.
Then there exists $\beta>1$ such that for all $u\in\mathbb{R}-\{0\}$
$$  \frac{r(u)}{u}\leq r'(u)\leq \beta \frac{r(u)}{u}.$$  
\end{theorem}

\begin{proof}
 If $u$ is positive, by the Mean Value Theorem there exists $c\in(0,u)$ such that $r'( c)=\frac{r(u)}{u}$.  Then since $r'$ is a homeomorphism of $[0,\infty)$ and is therefore increasing on $[0,\infty)$, we have $r'(u)>r'(c )$.  Thus $r'(u)\geq \frac{r(u)}{u}$.  To achieve an upper bound we again use the Mean Value Theorem except this time on the interval $[u,2u]$.  This gives 
$$ r'(u)\leq \frac{r(2u)-r(u)}{u}\leq \beta\frac{r(u)-r(0)}{u}= \frac{\beta r(u)}{u}.$$
The second inequality holds since $\Phi$ is quasisymmetric, and as stated in the proof of Lemma \ref{Quasidistance}, we have $d((u,v),(u',v))=|u-u'|$.  Hence there exists some $\beta>1$ such that $r(2u)-r(u)\leq \beta (r(u)-r(0))$.  
\\
\indent The inequalities for negative $u$ are proved in a similar manner. \end{proof}

%%%%%%%%%%%%%% Analytic Defn of QS

\section{An Analytic Definition of Quasisymmetry}\label{AnalyticDefnSection}

In this section we will use conjugation by our quasisymmetry $\Phi$ to develop an analytic definition of quasisymmetry in the $r$-Grushin plane.
\\
 \indent For the next several results let $g=(g_1,g_2):G\to G$ be a homeomorphism.  We define $f=f_1+if_2:\mathbb{C}\to\mathbb{C}$ to be the conjugation of $g$ by $\Phi$.  In other words $f=\Phi\circ g\circ \Phi^{-1}$.  Let $U=\frac{\partial }{\partial u}$ and $V=r'(u)\frac{\partial}{\partial v}$ be the vector fields corresponding to our metric on the $r$-Grushin plane.  Recall in Definition \ref{rGrushinBeltramiEquation} we gave the notation $W=\frac{1}{2}(U-iV)$, $\overline{W}=\frac{1}{2}(U+iV)$ and $\tilde{g}=\Phi\circ g$.  
 \\ \indent The next lemma demonstrates a relationship between the classical Beltrami equation and the $r$-Grushin Beltrami equation both of which were defined in the introduction.  The following theorem is an analytic definition of quasisymmetry on the $r$-Grushin plane.

\begin{lemma}\label{BeltramiCoefs}  Suppose $f$ and $g$ have partial derivatives that exist almost everywhere.  Then $\tilde{g}$ satisfies the $r$-Grushin Beltrami equation if and only if $f$ satisfies the classical Beltrami equation.  The equations are stated in Definitions \ref{CAnalyticDefn} and \ref{rGrushinBeltramiEquation} respectively.
\end{lemma}

\begin{proof}
 Wherever our derivatives exist we have by the chain rule:
\begin{enumerate}
\item $\frac{\partial f_1}{\partial x}|_{\Phi(w)}=\frac{1}{r'(u)}U(r(g_1))|_w$
\item $\frac{\partial f_1}{\partial y}|_{\Phi(w)}=\frac{1}{r'(u)}V(r(g_1))|_w$
\item $\frac{\partial f_2}{\partial x}|_{\Phi(w)}=\frac{1}{r'(u)}U(g_2)|_w$
\item $\frac{\partial f_2}{\partial y}|_{\Phi(w)}=\frac{1}{r'(u)}V(g_2)|_w$
\end{enumerate}

 and therefore $$\mu\circ\Phi
	=\frac{ \frac{\partial f}{\partial \bar{z}}}{\frac{\partial f}{\partial z}}
	=\frac{U(r(g_1))-V(g_2)+i(U(g_2)+V(r(g_1)))}{U(r(g_1))+V(g_2)+i(U(g_2)-V(r(g_1)))}
	=\frac{\overline{W}\tilde{g}}{W\tilde{g}} \text{ a.e..} $$ 
\end{proof}

We require a definition of absolute continuity on lines in the $r$-Grushin plane before giving our theorem.

\begin{definition}
Suppose $g$ is a homeomorphism of the $r$-Grushin plane.  We say $\tilde{g}$ is absolutely continuous on a horizontal interval $I_v=\{(u,v): a<u<b\}$ if for every $\epsilon>0$ there exists a $\delta>0$ such that whenever $\{[w_{i},w_{i}']\}_{1\leq i\leq n}$ is a disjoint collection of sub-intervals of $I_v$ $$\Sigma_{i=1}^n |\Phi(w_{i}')-\Phi(w_{i})|<\delta \implies \Sigma_{i=1}^n|\tilde{g}(w_{i}')-\tilde{g}(w_{i})|<\epsilon.$$
We define absolute continuity on vertical line segments analogously.

The function $\tilde{g}$ is absolutely continuous on lines if for every rectangle $R=\{(u,v): a < u < b, c < v < d\}$, 
$\tilde{g}$ is absolutely continuous on a.e. horizontal interval $I_v=\{(u,v):a<u<b\}$ and a.e. vertical interval $I_u=\{(u,v): c<v<d\}$ where almost every is with respect to Lebesgue measure.

\end{definition}

We have defined absolute continuity on lines in this manner so that $f$ is absolutely continuous on lines exactly when $\tilde{g}$ is absolutely continuous on lines.  

 We now prove Theorem \ref{QuasisymAnalyticDefn}.

 \begin{proof}
 Suppose $g$ is quasisymmetric.  Then since $\Phi$ is quasisymmetric, it follows that $f$ is quasisymmetric and hence quasiconformal.  So by the analytic definition of quasiconformality, the partial derivatives of $f$ exist a.e. and where they exist
 $$ \frac{\partial f}{\partial \bar{z}}=\mu\frac{\partial f}{\partial z}$$ for some measurable $\mu$ with $||\mu||_{\infty}<1$.  Furthermore $f$ is absolutely continuous on lines, which implies $\tilde{g}$ is absolutely continuous on lines.  Since each component of $\Phi$ and $\Phi^{-1}$ is differentiable except at the vertical axis, the partial derivatives of $g$ exist a.e.  Therefore $\overline{W}\tilde{g}$ and $W\tilde{g}$ exist a.e. and by our lemma $$ \overline{W} \tilde{g}=(\mu\circ\Phi) W \tilde{g}.$$  Since $||\mu||_{\infty}<1$ we have $ ||\nu||_{\infty}=||\mu \circ \Phi||_{\infty}<1$.
 \\ \indent Now suppose $\tilde{g}$ is absolutely continuous on lines and satisfies equation \ref{rGrushinBeltrami} for all points at which it is defined.  Then $f$ is absolutely continuous on lines and hence has partial derivatives that exist a.e..  As in our proof of the forwards implication, this implies that the partial derivatives of $g$ exist a.e., and $g$ satisfies the $r$-Grushin Beltrami equation.  Therefore $f$ satisfies the classical Beltrami equation and is thus quasiconformal.  Finally by conjugation, $g$ is quasisymmetric.
 \end{proof}

One would like to be able to replace quasisymmetry with quasiconformality in this theorem.  It is a well known result that quasisymmetry implies quasiconformality \cite{MR1800917}.  However, the converse does not always occur, and so far we have been unable to either prove or disprove it for the $r$-Grushin plane.  A partial answer to our question is in Theorem \ref{AnalyticConformalDef}, where we will show that on certain domains $\nu$ being identically zero implies $g$ is conformal.  The limitations on the domain arise when $g$ does not preserve the singular line.  We will discuss this following Theorem \ref{AnalyticConformalDef}.

%%%%%%%%%%%%%%% Conformal Mappings on the Grushin Plane

\section{Conformal Mappings on the $r$-Grushin Planes}\label{ConformalSection}

Since conformal mappings play a vital role in the study of quasiconformal mappings, it is of interest to us to find a useful characterization of them on the $r$-Grushin plane.  We will first develop a definition of conformality on the $r$-Grushin plane from the definition of conformal mappings on Riemannian manifolds.  This is appropriate since the $r$-Grushin plane is Riemannian everywhere except on the singular line.  Throughout the rest of the section we will provide further justification for our definition by looking at the classical Beltrami definition of conformality, and an earlier paper by Payne \cite{MR2230585}.   
\\ \indent Let $M$ be a $C^{\infty}$ Riemannian manifold and $g$ be a homeomorphism from $M$ to $M$.  Recall $g$ is conformal if the pullback of the Riemannian metric by $g$ is equal to the metric multiplied by some positive function.  Since we assume $M$ is $C^{\infty}$, we also have $g$ is infinitely differentiable \cite{MR0442846}.  The length element for our metric on $G-\{u=0\}$ is
$$ du^2+\frac{dv^2}{(r'(u))^2} $$
and its pullback by a function $g:G\to G$
is 
$$ \left[(U(g_1))^2+\frac{(U(g_2))^2}{(r'(g_1))^2}\right]du^2
+\frac{1}{(r'(u))^2}\left[(V(g_1))^2+\frac{(V(g_2))^2}{(r'(g_1))^2}\right]dv^2
+\frac{2}{r'(u)}\left[U(g_1)V(g_1)+\frac{U(g_2)V(g_2)}{(r'(g_1))^2}\right]dudv.$$
Recall $r'$ is zero only at zero so these expressions make sense on $G-\{u=0\}$ whenever $g_1$ is also non-zero on this domain.  
Thus we define conformality on the $r$-Grushin plane as follows:

\begin{definition} Suppose $A$ and $B$ are domains in $G$ and $g=(g_1,g_2): A \to B$ is a homeomorphism.  Define $A'=A-\{u=0 \text{ or } g_1(u)=0\}$.  We say $g|_{A'}$ is conformal provided that 

$$D_r g=\left( \begin{array} {cc}
U(g_1) & V(g_1) \\
\frac{U(g_2)}{r'(g_1)} & \frac{V(g_2)}{r'(g_1)}
\end{array} \right)$$
is defined and is a conformal matrix for every point in $A'$.  We say $g$ is conformal on all of $A$ if $g$ is conformal on $A'$ and for all points $w_0\in A-A'$, $$\lim_{w\to w_0} D_rg(w)$$ is defined and non-zero.  We take the limit along all paths in $A'$.
\end{definition}

 At first it may be tempting to think that the conjugation by $\Phi$ of any conformal map in the complex plane should be a conformal map in the $r$-Grushin plane.  This is not quite true.  There are mappings that are conformal everywhere on the complex plane, but when conjugated by $\Phi$ are only conformal on domains limited by the singular line.   For example, consider a horizontal translation $f(x+iy)=x+a+iy$.  When we conjugate $f$ with $\Phi(u,v)=\frac{1}{2}u|u|+iv$ we obtain the mapping
$$g(u,v)=\left(\sqrt{2}\frac{\frac{1}{2}u|u|+a}{\sqrt{|\frac{1}{2}u|u|+a|}},v\right).$$
Notice $V(g_1)=\frac{U(g_2)}{r'(g_1)}=0$ and $\frac{V(g_2)}{r'(g_1)}=U(g_1)=\frac{|u|}{\sqrt{|u|u|+2a|}},$
 and thus for the classical Grushin plane $D_rg$ is singular exactly on the line $u=0$, and the pre-image under $g$ of the line $u=0$.
 Therefore $g$ is only conformal on the Grushin plane on a domain excluding the singular line and the pre-image of the singular line.  We will discuss what must happen for a homeomorphism to be conformal on the entire Grushin plane after the next theorem.  
 
 The following result shows that for most domains in $G$ our description of conformality matches with the classical Beltrami differential definition of conformality.
\\
\begin{theorem}\label{AnalyticConformalDef} Suppose $A$ and $B$ are domains in $G$, and $g: A \to B$ is an orientation-preserving homeomorphism.  Then $g$ is conformal on the domain $A'=A-\{(u,v): u=0 \text{ or } g_1(u,v)=0\}$ if and only if $g$ is quasisymmetric and the Beltrami differential $\nu$ is identically zero on $A'$.
\end{theorem} 

\begin{proof}
Suppose $g$ is conformal.  Then all the derivatives in the entries of $D_rg$ must exist.  Thus since $g$ is orientation-preserving and $D_rg$ is conformal, we must have
 \begin{enumerate}
\item $U(r(g_1))=V(g_2)$ and $V(r(g_1))=-U(g_2)$, and
\item $D_r g$ is non-singular.
\end{enumerate}
Then by condition (1), $\overline{W} \tilde{g}=0$ which implies $\nu=0$ and by Lemma \ref{BeltramiCoefs}, $\mu=0$.  So we have $\frac{\partial f_1}{\partial x}=\frac{\partial f_2}{\partial y}$ and $\frac{\partial f_1}{\partial y}=-\frac{\partial f_2}{\partial x}$.  Furthermore
$$ J(D_rg)=\frac{1}{r'(g_1)}(U(g_1)V(g_2)-U(g_2)V(g_1)) \text{ and } 
 J(Df)=\frac{r'(g_1)}{(r'(u))^2}(U(g_1)V(g_2)-U(g_2)V(g_1)) |_{\Phi^{-1}(x,y)}.$$
 Thus since $r'$ takes the value zero only at $u=0$, $J(D_rg)$ and $J(Df)\circ\Phi$ are zero for the exact same values on $A'$.  So $Df$ is conformal almost everywhere and therefore $f$ is conformal and hence quasisymmetric.  Finally, since compositions of quasisymmetric maps are quasisymmetric, $g$ is quasisymmetric.
\\     \indent Now we assume $g$ is quasisymmetric and $\nu$ is identically zero on $A'$.  Since $\nu$ is identically zero on $A'$, we must have $\overline{W} \tilde{g}=0$, which implies condition (1).  Also since $g$ is quasisymmetric, $f$ is quasisymmetric and hence quasiconformal.  By Lemma \ref{BeltramiCoefs}, $\mu=0$ and hence $f$ is conformal.  Thus we can conclude that condition (2) also holds by our earlier statement regarding the Jacobians of $D_rg$ and $Df$, and therefore $g$ is conformal.
 \end{proof}

The situation is more complicated if we include the singular line and its pre-image in our domain.  For $g$ to be conformal in such a domain, 
$$J(D_rg)=\frac{r'(u)}{r'(g_1)}\left(\frac{\partial g_1}{\partial u}\frac{\partial g_2}{\partial v}-\frac{\partial g_2}{\partial u}\frac{\partial g_1}{\partial v}\right)$$
 must be non-singular.  Since we assume $g$ is orientation-preserving, this occurs exactly when 
 $$\lim_{u\to 0} \frac{r'(u)}{r'(g_1)}$$ is finite and non-zero.  Thus the singular line must map to itself.  
\\  
\indent This theorem also justifies our earlier work and in particular our selection of a relationship between the quasisymmetry $\Phi$ and the vector fields on $G$.  With other choices we do not have that $\overline{W}\tilde{g}=0$ when $g$ is conformal.  For example, if we use Meyerson's quasisymmetry 
$$ \Phi_M(u,v)=r_M(u)+iv=u|u|+iv$$
for the classical Grushin plane with vector fields $U=\frac{\partial}{\partial u}$ and $V=|u|\frac{\partial}{\partial v}$ we do not have that $|u|$ is equal $r_M'(u)$.  We compute 
$$\overline{W}\tilde{g}=\frac{1}{2}(U(g_1|g_1|)-V(g_2)+i(U(g_2)+V(g_1|g_1|))).$$
\indent Also we can use the same method as described at the beginning of this section to say, since $g$ is a homeomorphism on the classical Grushin plane, $g$ is conformal exactly when 
 $$\left( \begin{array} {cc}
 U(g_1) & V(g_1) \\
 \frac{U(g_2)}{|g_1|} & \frac{V(g_2)}{|g_1|}
 \end{array} \right)$$
 is a conformal matrix.  To simplify matters for the moment, we assume our domain does not include points on the singular line or points that map to the singular line.
Thus if $g$ is conformal, we must have $|g_1|U(g_1)=V(g_2)$ and $|g_1|V(g_1)=-U(g_2)$ which implies
$U(g_1|g_1|)=2V(g_2)$ and $V(g_1|g_1|)=-2U(g_2)$.  Hence we are not guaranteed that $\overline{W}\tilde{g}=0$ for conformal mappings.
\\ \indent To the best of the author's knowledge the only earlier discussion of conformal mappings on the Grushin plane is in a paper by Payne \cite{MR2230585}.  He defines a sequence of flows and states that the time-$s$ maps induced by the  solutions to any of the flows are conformal maps on the Grushin plane.  Here we will look at a generalization of Payne's flows and show that their solutions induce conformal maps on the $r$-Grushin plane.  In the following calculations $x$ and $y$ will be formal variables and $u$ and $v$ will be the Grushin coordinates as before.  First we define a sequence of functions of $x$ and $y$, $(\xi_k(x,y),\eta_k(x,y))$, $k\in\mathbb{N}$ by $(\xi_1,\eta_1)=(0,1)$, 
$$(\xi_2,\eta_2)=\left(\frac{r(x)}{r'(x)},y\right),$$

and the functions given inductively by
$$ (\xi_{k},\eta_{k})=(2\xi_{k-1}\eta_{k-1},\eta_{k-1}^2-(r'(x)\xi_{k-1})^2) \text{ for } k\geq 3.$$
The flows we will be solving are the autonomous differential equations:
$$\left(\frac{\partial{x}_k}{\partial s},\frac{\partial y_k}{\partial s}\right)=(\xi_k(x_k,y_k), \eta_k(x_k,y_k))$$
where $x_k=x_k(s,u,v)$ and $y_k=y_k(s,u,v)$ are functions of $u$, $v$ and a time parameter $s$.  We will let $g_k$ denote $(x_k,y_k)$,  In other words $g_k=(x_k,y_k):[0,\infty)\text{ x } G \to G$.  When $r=\frac{1}{2}u|u|$, these flows agree with Payne's flows up to a normalization.
We will show that each time-$s$ map associated with a solution with initial condition $x_k(0,u,v)=u$ and $y_k(0,u,v)=v$, is a conformal map on some domain in the $r$-Grushin plane.

One can easily compute the solutions to the first two flows
$g_1=(u,v+s)$, and
$g_2=(r^{-1}(r(u)e^s),ve^s),$
and check that the time-$s$ maps satisfy our definition of conformality.  The first solution gives vertical shifts by $s$.  In the classical Grushin plane ($r=\frac{1}{2}u|u|$) the second solution gives dilations by a factor of $e^{s/2}$.

To solve the remaining equations we will use the following auxiliary functions:
$$\Phi(x,y)=r(x)+iy \text{ and } b_k(x,y)=r'(x)\xi_k(x,y)+i\eta_k(x,y).$$
Recall $x$ and $y$ are formal variables.  We are interested in $b_k$ because 
\begin{equation}\label{difequation}
 b_k\circ g_k=\frac{\partial}{\partial s}(\Phi\circ g_k).
 \end{equation}
We will then find a non-iterative way of expressing $b_k(x,y)$ for each $k$ value and finally integrate $b_k\circ g_k$ to solve for $\Phi\circ g_k$.  We choose to solve for $\Phi\circ g_k$ instead of solving for $g_k$ directly, because this is a far easier task as will be evident when the reader sees the solutions in a moment.
One can compute $$b_k(x,y)=-i(b_{k-1}(x,y))^2 \text{ for } k\geq 4$$
and $$b_3(x,y)=-i\Phi(x,y)^2$$
by applying the definitions of $b_k$, $\xi_k$ and $\eta_k$.
Thus by induction we obtain
$$b_k(x,y)=i(-i\Phi(x,y))^{\alpha} \text{ where } \alpha =2^{k-2}.$$
Then by equation \ref{difequation} we have the following differential equations
$$\frac{\partial}{\partial s}(\Phi\circ g_k)=i(-i(\Phi\circ g_k))^{\alpha}.$$
Recall our initial condition on $g_k=(x_k,y_k)$ was $x_k(0,u,v)=u$ and $y_k(0,u,v)=v$.  So our initial condition is now $b(x_k(0,u,v),y_k(0,u,v))=r(u)+iv$.
Finally we obtain the solutions 
$$\Phi\circ g_k(s,u,v)=\frac{r(u)+iv}{([1-\alpha][-i(r(u)+iv)]^{\alpha-1}s+1)^{\frac{1}{\alpha-1}}}.$$

Let $g_k^s:G\to G$ denote the map $g_k$ for some fixed time $s$.  We will show for all $k\in\mathbb{N}$ and all $s\in [0,\infty)$, $g_k^s$ is conformal on some domain in the $r$-Grushin plane.  For $k\in\{1,2\}$, $g_k^s$ is conformal on the entire plane.  For $k\geq 4$, $g_k^s$ is conformal on some domain limited by a branch cut.  For example when $k=4$ if we specify that $-3i(r(u)+iv)^3s+1\in\mathbb{C}-\{(u,v): u\leq 0, v=0\}$ then we find that $g_4^s$ is conformal on the domain $$G-\left\{\Phi^{-1}(z): \arg(z)\in\left\{\frac{\pi}{2},\frac{7\pi}{6},\frac{11\pi}{6}\right\}, |z|>\left(\frac{1}{3s}\right)^{1/3}\right\}.$$
In general $g_k^s$ will be conformal on a domain with $\alpha-1$ cuts when $k\geq 4$.  We will discuss the case of $k=3$ after we prove conformality.

To prove each $g_k^s$ for $k\geq 3$ is conformal we look at the function
$$f_k^s=\Phi\circ g_k^s \circ\Phi^{-1}(z)=\frac{z}{([1-\alpha][-iz]^{\alpha-1}s+1)^{\frac{1}{\alpha-1}}}.$$
Thus $$\frac{\partial f_k^s}{\partial \bar{z}}=0,$$
and hence $f_k^s$ is conformal.  Then by Lemma \ref{BeltramiCoefs}, Theorem \ref{QuasisymAnalyticDefn} and Theorem \ref{AnalyticConformalDef}, $g_k^s$ is conformal.
\\ \indent	Earlier we noted that in the classical Grushin plane $g_1^s$ and $g_2^s$ were the familiar Grushin translations and dilations.  Now we can see that $g_3^s$ comes from a composition of translations, dilations and an inversion.  We have
  
 $$ f_3^s=\Phi\circ g_3^s \circ\Phi^{-1}= \frac{z}{1+izs}=\overline{\lambda^{-1}}\circ\ I_E \circ \lambda$$	
  where $\lambda (z)=zs-i$, $\Phi(u,v)=\frac{1}{2}u|u|+iv$ and $I_E$ is the Euclidean inversion $z\to 1/\bar{z}$.

The family of maps generated by $f_3$ is not entirely satisfactory since as $s$ goes to infinity $f_3^s$ degenerates to the zero map.  The slightly different family of maps $f_3^{*}(z)=\frac{is+z}{1+isz}$, goes to an inversion map as $s$ goes to infinity which is the behavior we would expect.

\indent A natural question to ask at this point is whether Payne's family of conformal maps includes all conformal maps on the Grushin plane or in some way generates all conformal maps on the Grushin plane.  If this is not true, are there functions we could add to Payne's list to enable us to obtain all conformal maps?  One way to approach these questions would be to try to find a more complete version of Theorem \ref{AnalyticConformalDef}.  In other words try to determine under exactly which conditions the theorem holds when we allow our domain to include the singular line.  This would give a complete characterization of conformal mappings on the Grushin plane which then could be compared to Payne's class of maps.

\bibliographystyle{plain}
 \bibliography{ColleenAckermannsReferences}

}
\end{document}